\documentclass[12pt,a4paper,reqno]{amsart}\usepackage{amsfonts,amsthm,latexsym,amsmath,amssymb,amscd,amsmath, epsf}
\providecommand{\keywords}[1]{\textbf{\textit{Keywords:}} #1}

\usepackage[utf8]{inputenc}
\usepackage{latexsym,amssymb,amsthm}
\usepackage{url,graphics} 
\usepackage[dvips]{graphicx,epsfig} 
\usepackage{relsize}
\usepackage{multicol}
\usepackage{color}
\usepackage{mathtools}
\usepackage{tikz}
\newcommand{\Span}{\textrm{span}}

\newcommand {\bR} {\mathbb {R}}
\newcommand {\bP} {\mathbb {P}}

\newcommand{\C} {\mathcal{C}}
 
\newcommand{\K} {\mathcal{K}} 
\newcommand{\F} {{\mathcal{F}}}

\newcommand{\Hilb} {{\mathcal{H}}}
\newcommand{\I} {\mathcal{I}}

\newcommand {\Ex} {{\mathcal{E}x}} %+ %F
\newcommand {\Cen} {{C}} %0 %T
\newcommand {\In} {{\mathcal{I}n}} %- %In

\newtheorem{theorem}{Theorem}
\newtheorem{proposition}[theorem]{Proposition}
 \newtheorem{definition}{Definition}

\newtheorem{corollary}{Corollary} \newtheorem{example}{Example}

\newtheorem{remark}{Remark} 

\newtheorem{question}{Question}
\newtheorem{conjecture}[question]{Conjecture}

\usepackage{lipsum}
\makeatletter
\DeclareRobustCommand{\cev}[1]{%
  \mathpalette\do@cev{#1}%
}
\newcommand{\do@cev}[2]{%
  \fix@cev{#1}{+}%
  \reflectbox{$\m@th#1\vec{\reflectbox{$\fix@cev{#1}{-}\m@th#1#2\fix@cev{#1}{+}$}}$}%
  \fix@cev{#1}{-}%
}
\newcommand{\fix@cev}[2]{%
  \ifx#1\displaystyle
    \mkern#23mu
  \else
    \ifx#1\textstyle
      \mkern#23mu
    \else
      \ifx#1\scriptstyle
        \mkern#22mu
      \else
        \mkern#22mu
      \fi
    \fi
  \fi
}

\begin{document}
          \numberwithin{equation}{section}

          \title[Classification of external Zonotopal algebras]{Classification of external Zonotopal algebras}

\author[G.~Nenashev]{Gleb Nenashev}
\address{ Department of Mathematics,
   Stockholm University,
   S-10691, Stockholm, Sweden}
\email{nenashev@math.su.se}

\begin{abstract}
In this paper we work with power algebras associated to hyperplane arrangements. There are three main types of these algebras, namely, external, central, and internal zonotopal algebras. We classify all external algebras up to isomorphism in terms of zonotopes. Also, we prove that unimodular external zonotopal algebras are in one to one correspondence with regular matroids. For the case of central algebras we formulate a conjecture.
\end{abstract}


\maketitle

%%%%%%%%%%%%%%%%%%%%%%%%%%%%%%%%%%%%%%%%%%%%%%%%%%%%%%%%%%%%%%%%%%%%%%%%%%%
\section{Introduction}

In this paper we work with power algebras, which are quotients of polynomial rings by power ideals. We consider zonotopal ideals, which are associated to zonotopes.  These ideals were independently introduced in two different ways. There are three types of algebras. We will work with the definition from F.\,Ardila and A.\,Postnikov~\cite{AP}; it comes from algebras generated by the curvature forms of tautological Hermitian linear bundles~\cite{Ar,SS}, see also papers~\cite{Berget, BergetIn, Huang, KN1, KN2, NS, Ne, PS, PSS, Shan}, where people work with quotients algebras by these ideals. At the same time the definition and the name was established by O.\,Holtz and A.\,Ron~\cite{HR}; it comes from Box-Splines and from Dahmen-Micchelli space~\cite{AS, DM, CP}, see also the papers~\cite{BR,DR, HRX,LenzMat, LenzBS, LenzH, NR, SX}. 

\smallskip

%\ection{Definition of Zonotopal algebra and history}
Let $A\in \bR^{n\times m}$ be a matrix of rank $n$.
Denote by $y_1,\ldots,y_m\in \bR^n$ the columns and by $t_1,\ldots,t_n\in \bR^m$ the rows.
For a matrix $A$, we define the zonotope $$Z_A:=\bigoplus_{i\in [m]} [0,y_i]\subset \bR^n$$ as the Minkovskii sum of intervals $[0,y_i],\ i\in[m]$.
By $\F(A)$ we denote the set of facets of  $Z_A$.
For any facet $H\in \F(A),$ we define $m(H)$ as the number of non-zero coordinates of the vector $\eta_H A\in \bR^m,$ where $\eta_H\in \bR^n$ is a normal for $H$. 

Let $\C_A^{(k)}$ be the quotient algebra $$\C_A^{(k)}:=\bR[x_1,\ldots,x_n]/{\I^{(k)}_A},$$
where $\I^{(k)}_A$ is the {\it zonotopal ideal} generated by the polynomials
$$p_H^{(k)}=(\eta_h\cdot (x_1,\ldots,x_n))^{m(H)+k},\ \ H\in \F(A).$$

\smallskip

There are $3$ main cases, where $k=\pm 1$ and $0$; they were considered in~\cite{AP,HR}.
\begin{itemize}
\item $k=1:\ \ $ $\C_A^{\Ex}=\C_A^{(1)}$ is the {\it external zonotopal algebra} for $A$;
\item $k=0:\ \ $ $\C_A^{\Cen}=\C_A^{(0)}$ is the {\it central zonotopal algebra} for $A$;
\item $k=-1:\ \ $ $\C_A^{\In}=\C_A^{(-1)}$ is the {\it internal zonotopal algebra} for $A$.
\end{itemize}
\begin{remark}
The case $k>1$ is not ``zonotopal", because the ideal $\hat{\I}^{(k)}$ generated by
$$p_{h}=(h\cdot (x_1,\ldots,x_n))^{m(h)+k},\ \ h\in \bR^n$$
is different from $\I^{(k)}$. They coincide only for the case where $k\leq 1$.

In the case $k\leq -5$, Hilbert series is not a specialization of the corresponding Tutte polynomial, see~\cite{AP}.
\end{remark}

\begin{theorem}[cf. \cite{AP, Berget, HR, LenzH}, External~\cite{PSS}, Central for graphs~\cite{PS}]
\label{thm:hilb} For a matrix $A\in \bR^{n\times m},$ the Hilbert series of zonotopal algebras are given by
\begin{itemize}
\item $\Hilb(\C_A^{\Ex})=q^{m-n}T_A(1+q,\frac{1}{q})$;
\item $\Hilb(\C_A^{\Cen})=q^{m-n}T_A(1,\frac{1}{q})$;
\item $\Hilb(\C_A^{\In})=q^{m-n}T_A(0,\frac{1}{q})$,
\end{itemize}
where $T_A$ is the Tutte polynomial of the vector configuration of the columns of $A$ (i.e., vectors $y_1,\ldots,y_m$).
\end{theorem}

There are other definitions of external algebras from~\cite{PSS}.
Let $\Phi_m$ be the square-free commutative algebra generated by $\phi_i,\ i\in[m]$, i.e., with relations
$$\phi_i \phi_j=\phi_j \phi_i,\ i,j\in[m]\ \ \ \textrm{and}\ \ \  \phi_i^2=0,\ i\in[m].$$
% $\phi_i^2=0$) and $\Phi^{\Cen}_A$ be with cuts relations, i.e.,
%$$\prod_{i\in I}\phi_{i}=0\ in\ \Phi^{\Cen}_A,$$
%for any cut $I\subset [m]$.

\begin{theorem}[cf.~\cite{PSS}]
\label{thm:definitions}
The external algebra $\C_A^{\Ex}$  is isomorphic to the subalgebra of $\Phi_A^{\Ex}:=\Phi_m$  generated by
$$X_i:=t_i\cdot (\phi_1,\ldots,\phi_m),\ i \in[n].$$
\end{theorem}
In papers~\cite{KN1,KN2} we constructed the analogue of Theorem~\ref{thm:definitions} in the case of central and internal zonotopal algebras for totally unimodular matrices, see the definition below. %Furthermore, there we presented a new proof of Theorem~\ref{thm:hilb}.

\medskip

The main interesting examples of zonotopal algebras arise for totally unimodular matrices and for graphs.
The matrix $A$ is {\it totally unimodular} if  any its minors is equal to $\pm 1$ or $0$. In this case the total dimensions of the algebras have a nice interpretation.
\begin{theorem}[cf.~\cite{HR}]\label{thm:dim:regular} Let $A\in \bR^{n\times m}$ be a totally unimodular matrix of rank $n$. Then the total dimension
\begin{itemize}
\item $dim(\C^{\Ex}_A)$ is equal to the number of lattice points of $Z_A$;
\item $dim(\C^{\Cen}_A)$ is equal to the volume of $Z_A$;
\item $dim(\C^{\In}_A)$ is equal to the number of interior lattice points of $Z_A$.
\end{itemize}
\end{theorem} 

The main examples of totally unimodular matrices are graphs. Namely, let $G$ be a graph on $n$ vertices; then the incidence matrix of any orientation of $G$ is totally unimodular. To construct the zonotopal algebra, we should forget exactly one row for each connected component of $G$. These algebras are independent (up to isomorphism) of the choice of orientations and rows.

These graphical algebras were considered in~\cite{Huang, KN1, NS, Ne, PSS}. In the graphical case Theorem~\ref{thm:dim:regular} can be written in graph theory terminology.
\begin{theorem}[cf.~\cite{PS}]\label{thm:dim:graph} Let $G$ be a graph. Then the total dimension
\begin{itemize}
\item $dim(\C^{\Ex}_G)$ is equal to the number of forests in $G$;
\item $dim(\C^{\Cen}_G)$ is equal to the number of trees in $G$ (in the connected case).
\end{itemize}
\end{theorem} 

It is well-known that the number of lattice points (volume) of the corresponding zonotope and the
number of forests (trees) of a graph are the same, see for example~\cite{Be, KW} (Points of the zonotope correspond to score vectors).
\begin{example}
Let $G$ be graph on the vertex set $\{0,1,2\}$ with $4$ edges $$(0,1), (0,2), (1,2),\ \textrm{and}\ (1,2),$$
see fig.~\ref{fig:zonotopeC3+Edge}, left. Let us orient all edges in the sense of increasing corresponding number. Consider the incidence matrix after forgetting of $0$-th row

\[
A:=\begin{bmatrix}
   -1       & 0 & 1 & 1 \\
   0       & -1 & -1 & -1 
\end{bmatrix},\]
see zonotope $Z_A$ on fig.~\ref{fig:zonotopeC3+Edge}, right. 

 {\begin{figure}[htb!]
\centering
\includegraphics[scale=0.75]{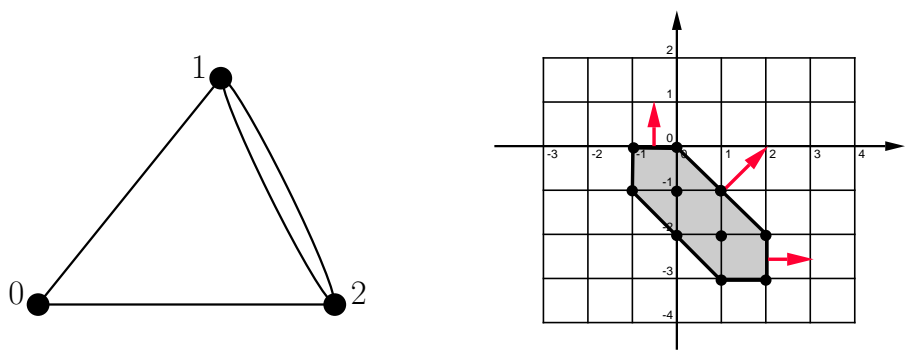}
\caption{A graph and its corresponding zonotope.}
\label{fig:zonotopeC3+Edge}
\end{figure}
}

The zonotope $Z_A$ has $6$ facets. We need the set of its normals (note that parallel facets have the same normal up to a factor). There are $3$ normals
\begin{itemize}
\item $\eta_1=(1,0);$
\item $\eta_2=(0,1);$
\item $\eta_3=(1,1).$
\end{itemize}

It is easy to check that $m(\eta_1)=m(\eta_2)=3$ and $m(\eta_3)=2$. Hence,
$$\I^{(k)}_A=\langle{x_1^{3+k},x_2^{3+k}, (x_1+x_2)^{2+k}}\rangle.$$
Then\begin{itemize}
\item $\Hilb(\C_A^{\Ex})= 1+2q+3q^2+3q^3+q^4$ \ \   and \ \  $dim(\C_A^{\Ex})=10$;  
\item $\Hilb(\C_A^{\Cen})= 1+2q+2q^2$ \ \   and  \ \  $dim(\C_A^{\Cen})=5$; 
\item $\Hilb(\C_A^{\In})= 1+q$ \ \   and \ \   $dim(\C_A^{\In})=2$. 
\end{itemize}
It is easy to check that $10,\ 5,\ \textrm{and}\ 2$ are exactly the number of lattice points, the area, and the number of interior lattice points of $Z_A$, respectively. Furthermore, $10$ and $5$ are the number of forests and trees in $G$.
In this case the Tutte polynomial is given by

$$T_G(x,y)=T_A(x,y)=x+y+x^2+xy+y^2,$$
Anyone can check the formulas for the Hilbert series.
\end{example}

\medskip

The following important property of external graphical algebras was proved in~\cite{Ne}.
\begin{theorem}[cf.~\cite{Ne}] Given two graphs $G_1$ and $G_2$. Then the following are equivalent:
\begin{itemize}
\item $\C_{G_1}^{\Ex}$ and $\C_{G_2}^{\Ex}$ are isomorphic as non-graded algebras;
\item $\C_{G_1}^{\Ex}$ and $\C_{G_2}^{\Ex}$ are isomorphic as graded algebras;
\item the graphical matroids $M_{G_1}$ and $M_{G_2}$ are isomorphic.
\end{itemize} 
\end{theorem}
The following conjecture was formulated for the central case:
\begin{conjecture}[cf.~\cite{Ne}] \label{conj:graphs} Given two connected graphs $G_1$ and $G_2$, the following are equivalent:
\begin{itemize}
\item $\C_{G_1}^{\Cen}$ and $\C_{G_2}^{\Cen}$ are isomorphic as non-graded algebras;
\item $\C_{G_1}^{\Cen}$ and $\C_{G_2}^{\Cen}$ are isomorphic as graded algebras;
\item the bridge-free matroids $M_{G_1}$ and $M_{G_2}$ are isomorphic.
\end{itemize} 
Here the bridge-free matroid of a graph is its graphical matroid after deleting all bridges.
\end{conjecture}

In the paper~\cite{NS}, the $\K$-theoretic filtration was considered, see definition there. Denote by $\K_G^{\Ex}$ the $\K$-theoretic filtration of $\C_G^{\Ex.}$
\begin{theorem}[cf.~\cite{NS}]
Given two graphs $G_1$ and $G_2$ without isolated vertices, the filtered algebras $\K_{G_1}^{\Ex}$ and $\K_{G_2}^{\Ex}$ are isomorphic if and only if $G_1$ and $G_2$ are isomorphic. 
\end{theorem}

\medskip

The structure of this paper is as follows: in \S~2 we present a classification of external zonotopal algebras and a conjecture for the central case; in \S~3 we prove our classification.

\medskip

\subsection*{Acknowledgments}
This material is based upon work supported by the National Science Foundation under Grant  DMS-1440140 while the author was staying at the Mathematical Sciences Research Institute in Berkeley, California, during the program ``Geometric and Topological Combinatorics'' in the fall~2017. He also would like to thank the participiants of the PA-seminar for their comments.

%%%%%%%%%%%%%%%%%%%%%%%%%%%%%%%%%%%%%%%%%%%%%%%%%%%%%%%%%%%%%%%%%%%%%%%%%%%
 \bigskip

\section{Main results}
\begin{definition}
Two linear spaces $V_1 \subset \bR^{m_1}$ and $V_2 \subset \bR^{m_2}$ are called  $z$-equivalent  if $m_1=m_2=m$ and there is an invertible diagonal matrix  $D\in \bR^{m\times m}$ and a permutation $\pi\in S_m$ such that
$$V_1=V_2 (\pi D).$$

The matrices $A_1 \in \bR^{n_1\times m_1}$  of rank $n_1$ and $A_2 \in \bR^{n_2 \times m_2}$ of rank $n_2$  are called $z$-equivalent  if the span of rows of $A_1$ is $z$-equivalent to the span of rows of $A_2$.
\end{definition}
\begin{remark}
It is easy to see that $z$-equivalence is an equivalence relation.

In the case when $A_1$ and $A_2$ do not have proportional columns, we can say that the matrix $A_1$ is equivalent to $A_2$ if and only if their zonotopes are equivalent (since we can reconstruct the ``matrix'' from the zonotope in this case).
\end{remark}
This equivalence is weaker than that of matroids. 
\begin{proposition}
If two matrices ${A_1}$ and ${A_2}$ are $z$-equivalent, then the matroids $M_{A_1}$ and $M_{A_2}$ are isomorphic.  
\end{proposition}

\medskip

It is easy to check that $\C_{A_1}^{\Ex}$ and $\C_{A_2}^{\Ex}$ are isomorphic if ${A_1}$ and ${A_2}$ are $z$-equivalent. The converse also holds.
\begin{theorem}
\label{thm:IsomZ}
Let $A_1 \in \bR^{n_1\times m_1}$ and $A_2 \in \bR^{n_2\times m_2}$ be two matrices of rank $n_1$ and $n_2$ respectively.
Then the following are equivalent:
\begin{itemize}
\item $\C_{A_1}^{\Ex}$ and $\C_{A_2}^{\Ex}$ are isomorphic as non-graded algebras;
\item $\C_{A_1}^{\Ex}$ and $\C_{A_2}^{\Ex}$  are isomorphic as graded algebras;
\item ${A_1}$ and ${A_2}$ are $z$-equivalent.
\end{itemize} 
\end{theorem}
\begin{corollary}
Let $A_1 \in \bR^{n_1\times m_1}$ and $A_2 \in \bR^{n_2\times m_2}$ be two matrices of ranks $n_1$ and $n_2$ respectively, with isomorphic external algebras $\C_{A_1}^{\Ex}\cong\C_{A_2}^{\Ex}$. Then the matroids $M_{A_1}$ and $M_{A_2}$ are isomorphic. 
\end{corollary}

The following theorems shows that unimodular external zonotopal algebras are in one to one correspondence with regular matroids. 
\begin{theorem}
\label{thm:IsomregZ}
Let $A_1 \in \bR^{n_1\times m_1}$ and $A_2 \in \bR^{n_2\times m_2}$ be two unimodular matrices of rank $n_1$ and $n_2$ respectively.
Then the following are equivalent:
\begin{itemize}
\item $\C_{A_1}^{\Ex}$ and $\C_{A_2}^{\Ex}$ are isomorphic as non-graded algebras;
\item $\C_{A_1}^{\Ex}$ and $\C_{A_2}^{\Ex}$ are isomorphic as graded algebras;
\item ${A_1}$ and ${A_2}$ are $z$-equivalent.
\item the matroids $M_{A_1}$ and $M_{A_2}$ are isomorphic.
\end{itemize} 
\end{theorem}
Since for graphs we have a totally unimodular matrix, all graphical matroids are regular; the converse is almost true. Every regular matroid may be constructed by combining graphic matroids, co-graphic matroids, and a certain ten-element matroid $R_{10}$, see~\cite{Seymour} or the book~\cite{Oxley}. In the graphical case the last theorem says that the algebra remembers graph up to $2$-isomorphism, see~\cite{Whitney}.

For the central case, we can extend Conjecture~\ref{conj:graphs} for all matrices. For a matrix $A$, we say that a column is a {\it bridge-column} if after deleting it the rank decreases.

\begin{conjecture}
\label{conj:zonotopal}
Let $A_1 \in \bR^{n_1\times m_1}$ and $A_2 \in \bR^{n_2\times m_2}$ be two matrices of ranks $n_1$ and $n_2$ respectively.
Then the following are equivalent:
\begin{itemize}
\item $\C_{A_1}^{\Cen}$ and $\C_{A_2}^{\Cen}$ are isomorphic as non-graded algebras;
\item $\C_{A_1}^{\Cen}$ and $\C_{A_2}^{\Cen}$ are isomorphic as graded algebras;
\item ${A_1'}$ and ${A_2'}$ are $z$-equivalent, where $A'_i\in \bR^{(n_i-k_i)\times (n_i-k_i)}$ is the submatrix of $A_i$ resulting after deleting all $k_i$ bridge-columns and those $k_i$ rows such that $rk(A_i')=n_i-k_i$.
\end{itemize} 
\end{conjecture}

%%%%%%%%%%%%%%%%%%%%%%%%%%%%%%%%%%%%%%%%%%%%%%%%%%%%%%%%%%%%%%%%%%%%%%%%%%%
\bigskip

\section{Proofs}

Let $B$ be a finite dimension algebra over $\bR$. We say that an element $r=\sum_{i=1}^k a_{2i}a_{2i+1}$ is {\it reducible} if $a_i\in B, i\in [2k]$ are nilpotent elements.

For a nilpotent  element $a\in B$ we define the {\it length} $\ell(a)$ as the maximal $\ell$ such that $a^{\ell}\neq 0$.

\begin{proof}[Proof of Theorem~\ref{thm:IsomZ}] Clearly, we have
$1\Longleftarrow 2 \Longleftarrow 3$, so we will prove $1\Longrightarrow 3$.
Let $\C_A^{\Ex}$ be our algebra. We will work with the square-free definition, i.e., $\C_A^{\Ex}$ is a subalgebra of $\Phi_m$, where 
$$m=max(\ell(a):\ a\in \C_A^{\Ex}).$$
 (note that we can work with $\Phi_m$ only theoretically, i.e., we do not know this embedding). 
We know which element is the unit, so we can chose basis  $x_1,\ldots, x_n$ of nilpotents of $\C_A^{\Ex}$ with the following property:
$$\ell(x+r)\geq \ell(x),$$ for any reducible $r\in\C_A^{\Ex}$ and $x\in\Span\{x_1,\ldots,x_n\}.$ Since we can define the algebra via some matrix $A$, then there is such basis.

\smallskip

Any element has the representation
 $$x_i=\sum_{k=1}^m a_{i,k}\phi_i + r_i,$$
 where $a_{i,k}\in \bR$ and $r_i$ is reducible. Let $A'=\{a_{i,k}:\ (i,k)\in[n]\times[m]\}$ be the corresponding matrix. Our goal is to reconstruct $A'$ up to $z$-equivalece.
 
 Consider the projective space $\bP^{n-1}$ over $\bR$. To finish the proof we should find the multiset $$\mathcal{A}:=\{(a_{1,k},a_{2,k},\ldots ,a_{n,k})\in \bP^{n-1},\ k\in[m]\}.$$

Define the set $S$ of all non-zeroes $s\in \bR$ such that there are $i\neq j\in[n]$ and a non-zero $t\in \bR$ for which
$$\ell(x_i-sx_j)<\ell(x_i-tx_j).$$
It is easy to see that $S$ is exactly the set 
$$\left\{\frac{a_{i,k}}{a_{j,k}}:\ i,j\in [n],\ k\in[m],\ \text{and}\ a_{i,k}, a_{j,k}\neq 0\right\}.$$
Then $S$ is a finite set.
  Define a theoretical set $\mathcal{S}$ of rows of $X$ as  
  $$\mathcal{S}:=\{(s_1,s_2,\ldots ,s_n)\in \bP^{n-1},\ s_i\in S\cup\{0\} \}.$$

Any element of $\mathcal{A}$ is an element of $\mathcal{S}$, so it is enough to find the multiplicity of any $s\in \mathcal{S}$. 

Consider the following partial order on elements of $\bP^{n-1}$:
  $$(p_1,p_2,\ldots ,p_n)\geq (p_1',p_2',\ldots ,p_n')$$
 if there is $t\in \bR$ such that
 $$\forall i\in [n],\ p'_i = \begin{cases} tp_i\\
 0. \end{cases}$$
Note that if, for any $s$, we know the summary multiplication of all $s'\geq s$ in $\mathcal{A}$, then we can calculate multiplicity of all elements.

Given $s\in \mathcal{S},$ then the summary multiplication of all $s'\geq s$ is equal to $$\ell\left(\sum_{i\in I} b_i x_i\right)-\ell\left(\sum_{i\in I} c_i x_i\right),$$
where \begin{itemize}
 \item $I\subseteq [n]$ is the support of $s$;
 \item $b_i,\ i\in I$ are generic;
 \item $c_i,\ i\in I$ are generic with the linear condition $\sum_{i\in I}{c_i}{s_i}=0$.
  \end{itemize}
  Let us check it:
 \begin{multline*}\ell\left(\sum_{i\in I} b_i x_i\right)=
  \ell\left(\sum_{i\in I} b_i (\sum_{k=1}^m a_{i,k}\phi_k+r_k)\right)=\ell\left(\sum_{i\in I} b_i (\sum_{k=1}^m a_{i,k}\phi_k)\right)=\\
 = \ell\left(\sum_{k=1}^m (\sum_{i\in I} b_ia_{i,k})\phi_k\right)=\#\{i\in k:\  \sum_{i\in I} b_ia_{i,k}\neq0\},
   \end{multline*}
   Similarly we have
   $$\ell\left(\sum_{i\in I} c_i x_i\right)=\#\{i\in k:\  \sum_{i\in I} c_ia_{i,k}\neq0\}.$$
   
Since $b_i,\ i\in I$ and $c_i,\ i\in I$ are generic with one condition $\sum_{i\in I}{c_i}{s_i}=0$, we have the following property:
 if $\sum_{i\in I} b_ia_{i,k}\neq 0$ then 
$\sum_{i\in I} c_ia_{i,k}=0$ if and only if $(a_{1,k},\ldots,a_{n,k})\geq s$.

Hence, we can compute the multiplicity of any $s$.    \end{proof}

\begin{proof}[Proof of Theorem~\ref{thm:IsomregZ}]
We know $1\Longleftrightarrow 2 \Longleftrightarrow 3 \Longrightarrow 4$ by Theorem~\ref{thm:IsomZ}, where we  also reconstructed a matrix, so we know the matroid.

\smallskip

 $3 \Longleftarrow 4$. Let $A_1$ and $A_2$ be two totally unimodular matrices which give the same regular matroid (we assume that the order of elements are the same). 
  
Also if $M$ is a regular matroid, then all orientations of $M$ differ only by reorientations (see Corollary~7.9.4 ~\cite{Oriented}). Hence, we can multiply some columns of $A_2$ by $-1$ and get $A_2'$ such that $A_1$ and $A_2'$ have the same oriented matroid.
 
 It is well-known that if we have a totally unimodular matrix $A_i$, then all minimal linear dependents of its columns have coefficients $\pm 1$. We get that matrices $A_1$ and $A_2'$ have linear dependents with the same coefficients and, hence, $A_1\sim_z A_2'\sim_z A_2$.
\end{proof}


\begin{thebibliography}{8}

\bibitem{AS} A.\,A.\,Akopyan, A.\,A.\,Saakyan, 
{\it A system of differential equations that is related to the polynomial class of translates of a box spline,}  Mat. Zametki 44:6 (1988),  pp. 705--724

\bibitem{AP} F.\,Ardila, A.\,Postnikov, 
{\it Combinatorics and geometry of power ideals,} Trans. of the AMS 362:8 (2010), pp 4357--4384.


\bibitem{Ar} V.\,I.\,Arnold,
{\it Remarks on eigenvalues and eigenvectors of {H}ermitian matrices, {B}erry phase, adiabatic connections and quantum {H}all effect}, Selecta Mathematica 1:1 (1995), 1--19
	
	
\bibitem{Berget}
A.\,Berget, {\it Products of linear forms and {T}utte polynomials,}
{{European~J. Combin.}} {31} (2010), pp. 1924--1935.

\bibitem{BergetIn}
A.\,Berget, {\it Internal Zonotopal Algebras and the Monomial Reflection Groups,} \url{https://arxiv.org/abs/1611.06446}

\bibitem{Be} O.\,Bernardi, {\it Tutte polynomial, subgraphs, orientations and sandpile model: new connections via embeddings,} Electronic J. Combinatorics 15:1 (2008).

\bibitem{Oriented} A.\,Björner, M.\,Las Vergnas, B.\,Sturmfels, N.\,White, G.\,M.\,Ziegler, {\it Oriented matroids,} volume 46 of Encyclopedia of Mathematics and its Applications (1999).







\bibitem{DM} W.\,Dahmen, A.\,Micchelli, {\it On the local linear independence of translates of a box spline,}  Studia Math. 82:3 (1985), pp.243--263

\bibitem{BR} C.\,De Boor, A.\,Ron, {\it On polynomial ideals of finite codimension with applications to box spline theory}, J. Math. Anal. Appl. 158:1 (1991), pp. 168--193

\bibitem{CP} C.\,De Concini, C.\,Procesi, {\it Topics in hyperplane arrangements,} polytopes and
box-splines, Universitext, Springer, New York, 2011.

\bibitem{DR} N.\,Dyn, A.\,Ron, {\it Local approximation by certain spaces of exponential polynomials,
approximation order of exponential box splines, and related interpolation problems,} Trans.
Amer. Math. Soc. 319 (1990), no. 1, 381–403.


 
\bibitem{HR} O.\,Holtz, A.\,Ron, {\it Zonotopal algebra,} Advances in Mathematics 227 (2011), pp 847--894.

\bibitem{HRX} O.\,Holtz, A.\,Ron, Z.\,Xu, {\it Hierarchical zonotopal,} spaces.Transactions of the American Mathematical Society 364.2 (2012): 745-766.

\bibitem{Huang} B.\,Huang, {\it Monomization of Power Ideals and Generalized Parking
Functions,} \url{https://math.mit.edu/research/highschool/primes/materials/2014/Huang.pdf}

\bibitem{KN1} A.\,N.\,Kirillov, G.\,Nenashev, {\it On $Q$-deformations of Postnikov-Shapiro algebras,} Séminaire Lotharingien de Combinatoire, 78B.55, FPSAC (2017) 12 pp.

\bibitem{KN2} A.\,N.\,Kirillov, G.~Nenashev, {\it Unimodular zonotopal algebra,} in preparation

\bibitem{KW} D.\,J.\,Kleitman, K.\,J.\,Winston, {\it Forests and score vectors,} Combinatorica 1:1 (1981), pp 49--54.

\bibitem{LenzMat} M.\,Lenz,  {\it Zonotopal algebra and forward exchange matroids.} Advances in Mathematics 294 (2016), 819--852. 



\bibitem{LenzBS} M.\,Lenz,  {\it Lattice Points in Polytopes, Box Splines, and Todd Operators.} International Mathematics Research Notices 2015:14 (2016),  5289--5310. 

\bibitem{LenzH} M.\,Lenz, {\it Hierarchical Zonotopal Power Ideals,} 
European Journal of Combinatorics 33 (2012), pp. 1120-1141.

\bibitem{NR} L.\,Nan, A.\,Ron , {\it External zonotopal algebra,} Journal of Algebra and its Applications 13.02 (2014):1350097.

\bibitem{NS} G.\,Nenashev, B.\,Shapiro, {\it``K-theoretic'' analog of Postnikov-Shapiro algebra distinguishes graphs,} Journal of Combinatorial Theory, Series A, 148 (2017), pp. 316--332.

\bibitem{Ne} G.\,Nenashev, {\it Postnikov-Shapiro algebras, graphical matroids and their generalizations,}   \url{https://arxiv.org/abs/1509.08736}%{arXiv:1509.08736}.

%\bibitem{OTe}
%Orlik P., Terao H., {\it Commutative algebras for arrangements,} 
%Nagoya
% Math.~J. 134 (1994), 65--73.

\bibitem{Oxley} J.\,G.\,Oxley, {\it  Matroid theory,} Oxford University Press, USA, 2006. 

\bibitem{Seymour} P.\,D.\,Seymour, {\it Decomposition of regular matroids,} Journal of combinatorial theory, Series B, 28:3 (1980), pp.305--359.

\bibitem {PS} A.\,Postnikov, B.\,Shapiro, {\it Trees, parking functions, syzygies, and deformations of monomial ideals,} Trans. Amer. Math. Soc. 356:8 (2004),  pp 3109--3142.

\bibitem {PSS} A.\,Postnikov,  B.\,Shapiro, M.\,Shapiro, {\it Algebras of curvature forms on homogeneous manifolds,} Differential topology, infinite-dimensional Lie algebras, and applications, Amer. Math. Soc. Transl. Ser. 2, 194 (1999), pp 227--235.

\bibitem {SS} B.\,Shapiro, M.\,Shapiro,  {\it On the ring generated by Chern $2$-forms on ${\rm SL}_n/B$,} C. R. Acad. Sci. Paris. I Math. 326:1 (1998),  pp 75--80.

\bibitem{Shan} J.\,J.\,Shan, {\it A special case of Postnikov-Shapiro conjecture} 
\url{https://arxiv.org/abs/1307.5895}

\bibitem{SX} B.\,Sturmfels, Z.\,Xu, {\it Sagbi bases of Cox-Nagata rings,} J. Eur. Math. Soc. 12 (2010), no. 2, 429–459

\bibitem{Whitney} H.\,Whitney, {\it 2-isomorphic graphs,} American Journal of Mathematics 55 (1933),  pp. 245--254.


\end{thebibliography}
\end{document}